\documentclass[12pt]{amsart}
\usepackage{amsmath,amstext, amsfonts,amssymb}
\usepackage[T2A]{fontenc}
\usepackage[utf8]{inputenc}
\usepackage[english]{babel}
\usepackage{cite}
\usepackage[dvipsnames]{xcolor}
\usepackage{mathrsfs}

\newcommand{\RR}{\mathbb R}
\newcommand{\NN}{\mathbb N}

\newtheorem{theorem}{Theorem}
\newtheorem{lemma}{Lemma}
\newtheorem{remark}{Remark}
\newtheorem{definition}{Definition}
\newtheorem{corollary}{Corollary}
\newtheorem{problem}{Problem}

\allowdisplaybreaks

\begin{document}
\title[Some optimal recovery problems]{Some optimal recovery problems for operators on classes of $L$-space valued functions}
\author{V.~Babenko, V.~Kolesnyk, O.~Kovalenko}
\address[V.~Babenko, O.~Kovalenko]{Department of Mathematical Analysis and Theory of Functions, Oles Honchar Dnipro National University, Dnipro, Ukraine}
\email[V. Babenko]{babenko.vladislav@gmail.com }
\email[O. Kovalenko]{olegkovalenko90@gmail.com}
\address[V.~Kolesnyk]{Department of Mathematics and Computer Science, Drake University, Des Moines, USA}
\email[V. Kolesnyk]{vira.kolesnyk@drake.edu}
\begin{abstract}
We solve three optimal recovery problems for operators on classes of $L$--space (which is a semilinear metric space with two additional axioms that connect the metric with the algebraic operations) valued functions that are defined by a majorant of their modulus of continuity. Consideration of $L$--spaces valued functions allows to treat multi- and fuzzy-valued functions, as well as random processes and other non-real valued functions in a unified manner.
\end{abstract}


\maketitle
\section{Introduction} 
Let a metric space $(X, h_X)$, sets  $Z$, $Y$, a class of elements $W\subset Z$, as well as mappings $\Lambda \colon Z\to X$ and $I\colon W\to Y$  be given. We call an arbitrary mapping $\Phi\colon Y\to X$ a method of recovery of the mapping $\Lambda$ on the class $W$ based on the information given by the mapping $I$. The error of recovery of the mapping $\Lambda$ on the class $W$ by the method $\Phi$ based on the information given by the mapping $I$ is given by the formula
$$
{\mathcal E}(\Lambda,W,I,\Phi, X)={\sup\nolimits_{w\in W}h_X(\Lambda(w), \Phi(I(w)))}.
$$
The quantity 
\begin{equation}\label{errorOfRecovery}
    {\mathcal E}(\Lambda,W, I, X)=\inf\nolimits_{\Phi} {\mathcal E}(\Lambda,W,I,\Phi, X)
\end{equation}
is called the optimal error of recovery of the mapping $\Lambda$ on the class $W$ based on the information given by the mapping $I$.
The problem of optimal recovery  of the mapping $\Lambda$ on the class  $W$ with the information given by $I$ in the metric of the space $X$ is to find quantity~\eqref{errorOfRecovery}
and a method $\Phi^*$ (if such a method exists) on which the infimum on the right-hand side of~\eqref{errorOfRecovery} is attained.

In this article we solve three problems of optimal recovery on the classes $W = H^\omega(T,X)$ of functions $f\colon T\to X$ that take values in an $L$-space $X$ (i.e., a semilinear metric space with two additional axioms that connect the metric with the algebraic operations) and have a given majorant $\omega$ of their modulus of continuity; precise definitions will be given in the following sections.

Results on optimal recovery on classes $H^\omega(T,X)$ with $X = \RR$ can be found in~\cite{Korneichuk68,Motornyi,Sukharev, Babenko76, Babenko77, Babenko95,Chernaya1, Chernaya2}; some cases with $X\neq \RR$ can be found in~\cite{Drozhzhina,Babenko16a,BBPS,Kovalenko20}. In papers~\cite{Babenko19a,VeraBabenko_JANO,Babenko20,Babenko21,Babenko22,Babenko23} questions of optimal recovery of operators
acting on $L$-space-valued functions were considered.

The first of the three problems mentioned above is as follows:
\begin{problem}
Solve the problem
of optimal recovery of the operator
\begin{equation}\label{integralLambda}
\Lambda f = \int_Qf(s)d\mu(s)
\end{equation}
on the class $H^\omega(T,X)$ based on the information operator
\begin{equation}\label{informationOperator1}
I(f) = (P(f(x_1)),\ldots, P(f(x_n))),
\end{equation}
where $P$ is the convexifying operator that is used in the definition of the Lebesgue integral for an $L$-space-valued functions, see Section~\ref{s::integral}.
\end{problem}
This result generalizes~\cite{Lebed68, Korneichuk68,Babenko16a}.

Let $X$ be an $L$-space, and $T$ be a measurable space. We say that a function $f\colon T\to X$ is measurable, if for each $x\in X$, the real-valued function $h_X(f(\cdot), x)$ is measurable. 
By $ M(T,X)$ we denote the space of all measurable functions  $f\colon T\to X$.

We say that a function $f\in M(T,X)$ is bounded, if for some $C > 0$ and $x\in X$ one has  $h_X(f(t), x)<C$ for all  $t\in T$. We denote by $B(T,X)$ the set of all  bounded functions $f\in M(T,X)$. In the space $B(T,X)$ one can introduce the uniform metric $$h_{B(T,X)}(f,g) = \sup_{t\in T} h_X(f(t),g(t)).$$

Assume $Y$ is also an $L$-space, $\lambda\colon X\to Y$ is a positively homogeneous Lipschitz mapping and $\varphi\colon B(T,\RR)\to Y$ is a monotone mapping (in the sense that if $0\leq f(t)\leq g(t)$ for all $t\in T$, then $\varphi(f)\leq \varphi(g)$). An operator $\Lambda\colon B(T,X)\to Y$ is said to be of $(\lambda,\varphi)$-type, if the following properties hold:
\begin{equation}\label{lambdaPhiProp1}  
h_Y(\Lambda f,\Lambda g)\le \varphi(h_Y(\lambda\circ f,\lambda\circ g)),
\end{equation}
and for any non-negative function $u\in B(T,\RR)$, and element $x\in X$
\begin{equation}\label{lambdaPhiProp2}  
\Lambda (u(\cdot)x)=\varphi (u(\cdot))\lambda(x).
\end{equation}
The set $\mathcal{L}(\lambda,\varphi)$ of all $(\lambda,\varphi)$-type operators is a subset of more general classes of operators that satisfy analogues of properties~\eqref{lambdaPhiProp1} and~\eqref{lambdaPhiProp2}, which were considered in~\cite{Babenko21} and~\cite{Babenko24AbstractOstrowski}. As it will be easily seen later, integral operator~\eqref{integralLambda} is of $(\lambda,\varphi)$ type, where $\lambda$ is the convexifying operator, and $\varphi$ is the integral of real-valued functions.

The second problem is as follows.
\begin{problem}
Solve the optimal recovery problem for the operator
\begin{equation}\label{valueOperator}
   \Lambda(f) = \lambda(f(t)),
\end{equation}   
$t\in T$ is fixed, on the class $H^\omega(T,X)$ based on the information operator 
\begin{equation}\label{generalInformationalSet}
   I(f) 
   = 
   \left(\frac{\Lambda_1(f\cdot \chi_{1})}{\varphi(\chi_{1})}, \ldots, \frac{\Lambda_n(f\cdot \chi_{n})}{\varphi(\chi_{n})} \right), 
\end{equation}
where $\Lambda_i\in \mathcal{L}(\lambda,\varphi)$, $i=1,\ldots, n$, and $\chi_1, \ldots, \chi_n\in B(T,\RR)$ are non-negative functions such that the denominators are non-zero.
\end{problem}

In particular, if each of $\Lambda_i$ is integral operator~\eqref{integralLambda} and each $\chi_i$ is the characteristic function of a set $Q_i\subset T$, then the recovery is done based on the mean values of $f$ on the sets $Q_i$, $i=1,\ldots, n$. For optimal recovery of the integral for real-valued functions based on such type of information see~\cite{BabenkoSkorokhodov,Babenko08} and references therein. The recovery of the integral was also considered based on known mean values of the function of $d\geq 2$ variables over surfaces of dimension less than $d$; see e.g.~\cite{Babenko11a,Babenko11b,Babenko11c} and references therein.

% {\color{blue} Те, що я можу зразу пригадати з цього приводу, це результати по оптимізації так званих інтервальних квадратурних формул на різноманітних класах числовтих функцій. Посилання треба шукати. Також результати пр оптимізації таких формул на деяких класах многозначних функцій. Крім того є ряд результатів по оптимальному відновленню по усередненій інфоимації на коасах функцій багатьох змінних. Мабуть є смисл їх згадати. }{\color{red} Я надаслав тобі на пошту декілька файлів. Їх і посилання в них монжа використатм}

Finally, in the case, when $T\subset \RR^d$, $d\in\NN$, is a convex compact set, $B_1,\ldots, B_n$ are closed balls with small enough radius $\varepsilon>0$, and $\mu$ is the Lebesgue measure, we solve the following problem. 
\begin{problem}
Solve the problem of optimal recovery of the convexifying operator $P$ in the uniform metric, based on the information \begin{equation}\label{meanValuesInformation}
    I(f) 
    = 
   \left(\frac{1}{\mu(B_1)}\int_{B_1}f(s)d\mu(s), \ldots, \frac{1}{\mu(B_n)}\int_{B_n}f(s)d\mu(s) \right).
\end{equation}
\end{problem}
In the case $d = 1$ a solution to this optimal recovery problem is contained in~\cite{Babenko23}. 
To the best of our knowledge, this result with $d>1$ is new even for the case of $X= \RR$.

The article is organized as follows. In Section~\ref{s::LSpace} we give necessary definitions and facts related to the notion of an $L$-space. In Section~\ref{s::ostrowskiInequality} we prove several theorems, which on the one hand are key ingredients to prove the main results about the optimal recovery problems, and on the other hand seem to be interesting on their own. Finally, Section~\ref{s::optimalRecoveryProblems} contains the solutions for the announced above optimal recovery problems.

\section{$L$-spaces}\label{s::LSpace}
\subsection{Definitions}

\begin{definition}
A set $X$ is called a semilinear space, if  operations of addition of elements and their multiplication on real numbers are defined in $X$, and the following conditions are satisfied for all $x,y,z\in X$ and $\alpha,\beta\in\mathbb{R}$:
\begin{gather*}
x+y=y+x;
\\ x+(y+z)=(x+y)+z;
\\ \exists\; \theta\in X\colon x+\theta=x;
\\ \alpha(x+y)=\alpha x +\alpha y;
\\ \alpha(\beta x)=\left(\alpha\beta\right)x;
\\ 1\cdot x=x,\; 0\cdot x=\theta.
\end{gather*}
\end{definition}

\begin{definition}
We call an element $x\in X$ convex, if for all $\alpha,\beta\geq 0$, $(\alpha + \beta)x=\alpha x+ \beta x.$
Denote by $ X^{\rm c}$ the subspace of all convex elements of the space $X$.
\end{definition}
\begin{remark}
Some authors (see e.g.~\cite{Borisovich}) include into the axioms of a semi-linear space the requirement $X=X^{\rm c}$.
\end{remark}

\begin{definition}
A  semilinear space $X$ endowed with a metric $h_X$ is called an $L$-space, if it is complete, separable, and for all $x,y,z\in X$, and $\alpha\in\mathbb{R}$
$$
h_X(\alpha x,\alpha y)=|\alpha|h_X(x,y);
$$
\begin{equation}\label{ax::LSemiIsotropic}
h_X(x+z,y+z)\leq h_X(x,y).
\end{equation}
\end{definition}

\begin{remark}
It follows from the triangle inequality and~\eqref{ax::LSemiIsotropic} that 
\begin{equation}\label{distBetweenSums}
   \forall x,y,z,w\in X\;\;\; h_X(x+z,y+w)\leq h_X(x,y)+h_X(z,w).   
\end{equation}
\end{remark}

\begin{definition}
An $L$-space $X$ is called isotropic, if inequality~\eqref{ax::LSemiIsotropic} turns into equality for all $x,y,z\in X$.
\end{definition}

Next we list some of the examples of $L$-spaces. More examples can be found in~\cite{Babenko19a, Babenko20, Babenko21}. 
Arbitrary separable Banach space and arbitrary complete and separable quasilinear normed space (see~\cite{Aseev}) are $L$-spaces. The space  $\Omega(X)$ of non-empty compact subsets of a separable Banach space $X$ endowed with the usual Hausdorff metric, the space $\Omega_{\rm conv}(X)$ of convex elements from $\Omega(X)$, and spaces of fuzzy sets (see e.~g.~\cite{Diamond2}) are also examples of $L$-spaces. All $L$-spaces mentioned above are isotropic.

An example of a non-isotropic $L$-space can be built as follows. Let $X = [0,\infty)$, for $\lambda\in\RR$, $x,y\in X$ set $x\oplus y = \max\{x,y\}$, $\lambda\odot x = |\lambda|x$. Then $(X,\oplus,\odot)$ with the metric $h_X(x,y) = |x-y|$, $x,y\in X$, is a non-isotropic $L$-space.

\subsection{Approximation of Measurable Functions by Simple Ones}
Let $\left(T,\mu\right) = \left(T,\mathcal{F},\mu\right)$ be a measurable space with a complete finite measure $\mu$ and $(X, h_X)$ be a separable metric space. 
  From~\cite[Theorem~I.4.20]{Varga} it follows that $f\in M(T,X)$ if and only if for each open set $A\subset X$ one has $f^{-1}(A)\in\mathcal{F}$. Moreover, for two functions $f,g\in M(T,X)$, the function $t\mapsto h_X(f(t),g(t))$ is measurable according to~\cite[Theorem~I.4.22]{Varga}.

The following variant of Egorov's theorem is true, see e.g.~\cite[Theorem~I.4.18]{Varga}.
\begin{theorem}\label{th::Egorov}
    If a sequence of functions $f_i\in M(T,X)$ converges to  $f\colon T\to X$ almost everywhere as $i\to\infty$, then $f\in M(T,X)$. Moreover, for each  $\varepsilon>0$ there exists a set $T_\varepsilon\in \mathcal{F}$ such that  $\mu(T\setminus T_\varepsilon) < \varepsilon$ and $f_i$ converges to  $f$ uniformly on $T_\varepsilon$ as $i\to\infty$.  
\end{theorem}

\begin{definition}
A mapping $f\colon T\to X$ is called simple, if it has a finite number of values $\left\{f_k\right\}_{k=1}^n$ on pairwise disjoint measurable sets $\left\{T_k\right\}_{k=1}^n$, $n\in\NN$.
\end{definition}

The following lemma holds.
\begin{lemma}\label{l::approximationOfMeasurableFunctions} For each  $f\in M(T,X)$ there exists a sequence of simple functions that converges to $f$ almost everywhere. If  $f$ is bounded, then the sequence of simple functions can be chosen to be uniformly bounded.    
\end{lemma}
\begin{proof}
    Let $\{x_1,\ldots, x_n,\ldots\}$ be a countable dense in $X$ set. For each  $k\in \NN$ set 
    $A_{k,1} = \{t\in T\colon h_X(f(t), x_1)< 1/k\}$,  and for each  $i\geq 2$, $A_{k,i} = \{t\in T\colon h_X(f(t), x_i)< 1/k\}\setminus \bigcup_{s=1}^{i-1} A_{k,s}$. Then each of the sets $A_{k,i}$ is measurable, and for all  $k\in \NN$ the sets $A_{k,i}$, $i\in\NN$, constitute a partition of the set  $T$. Since for each  $k\in \NN$, $\sum_{i=1}^\infty \mu(A_{k,i})  =\mu (T) <\infty$, there exists  $N(k)\in\NN$ such that $\mu \left(B_k\right) < 2^{-k}$, where $B_k = \bigcup_{i=N(k)+1}^\infty A_{k,i}$. Set 
    $$f_k(t) = 
    \begin{cases}
    x_i, & t\in A_{k,i}, i \leq N(k),\\
    x_1, & t\in B_k.
    \end{cases}
    $$ 
    Then for each  $t\in T\setminus B_k$ one has $h_X(f(t), f_k(t))\leq 1/k$, $k\in\NN$. Set  $C_m = \bigcup_{s=m}^\infty B_s$, $m\in\NN$. For arbitrary  $k\in\NN$, all  $s\geq k$ and all  $t\in T\setminus C_k$ one has  $h_X(f(t), f_s(t))\leq 1/s$; this implies that the sequence of simple functions  $f_s$, $s\in\NN$, converges uniformly to  $f$ on the set $T\setminus C_k$. Since   $C_{m+1}\subset C_m$ for all  $m\in\NN$ and $\mu(C_m)\to 0$ as  $m\to\infty$, we obtain that the sequence $f_s$ converges to  $f$ almost everywhere and the first statement of the lemma is proved.

    If for some  $x\in X$ one has $C:=\sup_{t\in T}h_X(f(t),x)<\infty$, then for the considered sequence  $f_k$, $k\in\NN$, and arbitrary  $t\in T\setminus B_k$ one has
    $$
    h_X(f_k(t), x)\leq     h_X(f_k(t), f(t)) + h_X(f(t), x)\leq 1 +C.
    $$
    On the other hand, if  $t\in B_k$, then 
    $h_X(f_k(t), x)= h_X(x_1, x)$. Thus the estimate from above for the distance  $h_X(f_k(t), x)$ does not depend on  $k$ for all  $t\in T$, hence the sequence of functions  $f_i$ is uniformly bounded.
\end{proof}

\subsection{Integration in $L$-spaces}\label{s::integral}
Let $X$ and $Y$ be $L$-spaces.
\begin{definition}
An operator $\lambda \colon X\to Y$ is called Lipschitz if for all $x,y\in X$
$$
     h_Y(\lambda(x),\lambda(y))\le h_X(x,y).
$$
\end{definition}
\begin{definition}  (See \cite{Vahrameev,Aseev})
A surjective operator $P\colon X\to X^{\rm c}$ is called a convexifying operator, if it is Lipschitz and 
\begin{enumerate}
\item $P\circ P= P$;
\item $\forall x,y\in X$ and $\forall \alpha,\beta\in\RR$, $P(\alpha x + \beta y) = \alpha P(x)+ \beta P(y)$.
\end{enumerate}
\end{definition}

For example, the operator  ${\rm co}\colon\Omega(\mathbb{R}^m)\to \Omega(\mathbb{R}^m)$ that maps each $x\in \Omega(\mathbb{R}^m)$ to its convex hull ${\rm co}\,x$ is a convexifying operator.

\begin{remark}\label{r::P(x)}
From the definition of a convexifying operator it follows that  $P(x) = x$ for all $x\in X^{\rm c}$.
\end{remark}
Indeed, if $x\in X^{\rm c}$ and $y\in X$ is such that $P(y) = x$, then $P(x) = P(P(y)) = P(y) = x$.

For completeness, we present the definition and some of the properties of the Lebesgue integral for the functions $f\in B(T,X)$, where $T$ is a metric compact, see~\cite{Vahrameev} and~\cite[\S 5]{Aseev}. 

 The Lebesgue integral of a simple mapping $f\in B(T,X)$ with respect to a  Borel measure $\mu$ is by definition
$$
\int_T f(s)d\mu(t) :=\sum_{i=1}^n P(f_i)\mu(T_i).
$$

The following properties hold for all simple functions $f,g\in B(T,X)$.
\begin{enumerate}
\item For all $\alpha,\beta\in \RR$
\begin{equation}\nonumber
\int_T \left(\alpha f(t)+\beta g(t)\right) d\mu(t)=\alpha \int_T f(t) d\mu(t)+\beta \int_T g(t) d\mu(t).
\end{equation}

\item The function $t\mapsto h_X \left(f(t), g(t)\right)$ is integrable and
 $$
    h_X\left(\int_T f(t) d\mu(t), \int_T g(t) d\mu(t)\right)\leq \int_T h_X \left(f(t),g(t)\right)d\mu(t).
$$

\item The function $P(f(\cdot))$ is integrable and
\begin{equation}\nonumber
\int_T f(t) d\mu(t) = P\left( \int_T f(t) d\mu(t)\right) = \int_T P(f(t)) d\mu(t).
\end{equation}
\item For disjoint measurable sets $T_1$ and $T_2$ such that $T=T_1\cup T_2$
$$
\int_T f(t) d\mu(t)=\int_{T_1} f(t) d\mu(t)+\int_{T_2} f(t) d\mu(t).
$$

\end{enumerate}
By Lemma~\ref{l::approximationOfMeasurableFunctions}, any function $f\in B(T,X)$ is an almost everywhere  limit of a uniformly bounded sequence $\{f^k\}$ of simple functions.  Using Theorem~\ref{th::Egorov} and standard arguments, one can prove that the sequence $\left\{\int_Tf^k(t)d\mu(t)\right\}_{k\in\NN}$ is fundamental. By definition, the integral $\int_Tf(t)d\mu(t)$ of the function $f$ is set to be the limit of this sequence; the fact that the limit does not depend on the choice of the sequence $\{f^k\}$ is easy to verify.
 It is clear that properties 1-4 of the Lebesgue integral for simple functions hold for arbitrary functions from $B(T,X)$.
 
 Note that in the case of $X$ being a Banach space, the integral defined above becomes the Bochner integral, see~ \cite[Sections~3.7--3.8]{hille}; in the case $X=\Omega( {\RR^m})$, the integral coincides with the Aumann integral, see~\cite[Theorem~12]{Aseev}.
 
 \subsection{Some Properties of $L$-spaces}
\begin{definition}
We say that an element $x\in X$ is invertible, if there exists an element $x'\in X$ such that $x+x'=\theta$. In this case the element $x'$ is called the inverse to $x$. Denote by $X^{\rm inv}$ the set of all invertible elements of the space $X$. 
\end{definition}
It is easy to see that an inverse element is unique whenever it exists.

In the space $\Omega(X)$ any element of the form $\{x\},$ $x\in X$, is convex and invertible. 

Remark~\ref{r::P(x)} immediately implies the following lemma.
\begin{lemma}\label{l::convexifyingOperator}
If $P$ is a convexifying operator in an isotropic $L$-space $X$, then $P(\theta) = \theta$. If $x\in X^{{\rm inv}}$, then $P(x)\in X^{\rm inv}$ and $P(x')=(P(x))'$.
\end{lemma}

We also  need the following lemmas, see e.g.~\cite{VeraBabenko_JANO}.  We give their proofs for completeness.
\begin{lemma}\label{l::inverseIsConvex}
If $x\in X^{\rm inv}\cap X^{\rm c}$, then $x'\in X^{\rm inv}\cap X^{\rm c}$. 
\end{lemma}
\begin{proof}
The implication  $x\in X^{\rm inv}\implies x'\in  X^{\rm inv}$ is obvious. We prove that $x\in X^{\rm inv}\cap X^{\rm c}\implies x'\in X^{\rm c}$.  Let $\alpha, \beta\geq 0$.
\begin{multline*}
\theta 
=
0\cdot \theta
=
((\alpha+\beta)\cdot 0)\cdot \theta
=
(\alpha+\beta)\cdot( 0\cdot \theta)
=
(\alpha+\beta)\theta 
\\= 
(\alpha+\beta)(x+x')
=
(\alpha+\beta)x+(\alpha+\beta)x'
=\alpha x+\beta x+(\alpha+\beta)x'
\end{multline*}
i.e.,  $(\alpha+\beta)x'=(\alpha x+\beta x)'$. On the other hand,
$(\alpha x +\beta x)'=\alpha x'+\beta x'$, since
$$
\alpha x+\beta x+\alpha x'+\beta x'
=
\alpha(x+x')+\beta(x+x')
=
\theta.
$$
Due to uniqueness of the inverse element, we obtain
$\alpha x'+\beta x'=(\alpha+\beta)x'$
i.e.,  $x'\in X^{\rm c}$. 
\end{proof}
\begin{lemma}\label{l::distBetweenInverseElems} 
For any $ x \in X^{\rm inv}\cap X^c$, one has
$
h_X(x,x')=2h_X(x,\theta).
$
\end{lemma}
\begin{proof} We have 
\begin{multline*}
h_X(x,x')=h_X(x+\theta,x')=h_X(x+x+x',x')
\\ \leq 
h_X(x+x,\theta)=h_X(x+x,x'+x)
\leq 
h_X(x,x').
\end{multline*}
Hence all inequalities necessarily turn into equalities and thus
$$h(x,x') = h_X(x+x,\theta)=h_X(2x,\theta)=2h_X(x,\theta).$$
\end{proof}
\begin{definition}
For a set $T$, a function $f\colon T\to\RR$ and  an element $x\in X^{\rm inv}$ define the function $f_x\colon  T\to X$, 
$$
f_x(t) = f_+(t)\cdot x + f_-(t)\cdot x',
$$
where for real $\xi$, $\xi_\pm :=\max\{\pm \xi ,0\}$,  and $f_\pm(t) = (f(t))_\pm$.
\end{definition}
 Let a modulus of continuity $\omega$ (i.e.,  a  non-decreasing continuous semi-additive function that vanishes at $0$), a metric space $(T,r)$, and an $L$-space $X$ be given. 

\begin{definition}
   By $H^\omega(T,X)$ we denote the space of functions $f\colon T\to X$ such that for all $t,s\in T$ one has $h_X(f(s), f(t))\leq \omega(r(t,s))$. 
\end{definition}

The following lemma can be proved similarly to~\cite[Lemma~6]{Babenko23}.
\begin{lemma}\label{l::LspaceValuedFunc}
Let $X$ be an isotropic $L$-space and $f\in H^\omega (T,\RR)$. If $x\in X^{\rm inv}\cap X^{\rm c}$ is such that $h_X(x,\theta)\leq 1$, then   $f_x\in H^\omega(T,X)$ and 
\begin{multline*}
\int_T f_x(t)dt  
=
\int_T f_+(t)dt\cdot x + \int_T f_-(t)dt\cdot x'
\\=
\left(\int_T f(t)dt\right)_+\cdot x + \left(\int_T f(t)dt\right)_-\cdot x'.
\end{multline*}
If $f$ is non-negative on $T$, then it is enough to require $x\in X^{\rm c}$ instead of $x\in X^{\rm inv}\cap X^{\rm c}$.
\end{lemma}

\subsection{Operators of $(\lambda,\varphi)$-type}
In this subsection $X$ and $Y$ are $L$-spaces.
\begin{definition}
We call an operator $\lambda \colon X\to Y$ positively homogeneous, if for all $\alpha\geq 0$ and $x\in X$,
$$
  \lambda(\alpha\cdot x) = \alpha\cdot \lambda(x),
$$
in particular $\lambda(\theta) = \theta$.
Denote by $\mathcal{H}(X,Y)$ the set of all positively homogeneous Lipschitz operators $\lambda\colon X\to Y$.
\end{definition}

\begin{definition}
A functional $\varphi\colon B(T,\RR)\to \RR$ is called monotone, if $\varphi(f)\leq \varphi(g)$, provided  $0\leq f(t)\leq g(t)$ for all $t\in T$. Denote by $\mathcal{M}(T)$ the set of all monotone functionals $\varphi\colon B(T,\RR)\to \RR$.
\end{definition}

\begin{definition}
Let $\lambda\in \mathcal{H}(X,Y)$ and  $\varphi\in \mathcal{M}(T)$ be given. An operator $\Lambda\colon B(T,X)\to Y$ is said to be of $(\lambda, \varphi)$-type, if conditions~\eqref{lambdaPhiProp1} and~\eqref{lambdaPhiProp2} hold. Denote by ${\mathcal L}(\lambda,\varphi)$ the set of all $(\lambda, \varphi)$-type  operators.
\end{definition}
Below are two examples of $(\lambda,\varphi)$-type operators.
\begin{enumerate}
\item  The operator $\Lambda$ from~\eqref{integralLambda} 
 is of $(\lambda,\varphi)$-type, where $\lambda$ is the convexifying operator, and $\varphi$ is the Lebesgue integral for the real-valued functions.
\item Let $Y$ be an isotropic $L$-space, $y\in Y^c$ such that $h_Y(y,\theta) \leq 1$ be fixed, and $\|\cdot \|$ be some monotone norm in $B(T,\RR)$. The operator $\Lambda\colon B(T,X)\to Y$, 
\begin{equation}\label{nonIntegralLambda}
 \Lambda(f) = \|h_X(f(\cdot), \theta)\|\cdot y
\end{equation}
 is of $(\lambda,\varphi)$-type, where $\varphi(u) = \|u\|$ and $\lambda(x) = h_X(x,\theta)\cdot y$. 
\end{enumerate}
Note that operator~\eqref{nonIntegralLambda} is not additive, provided $y\neq \theta_Y$ and there exist non-negative functions $u,v\in B(T,\RR)$ such that $\|u+v\|\neq \|u\|+\|v\|$; hence for such norms it can not be reduced to integral operator~\eqref{integralLambda}.

\section{Deviation From a Single Measurement}\label{s::ostrowskiInequality}
Let $T$ be a metric compact with metric $r$, $\chi\in B(T,\RR)$ be a non-negative function, $\mu$ be a Borel measure on $T$,  $X, Y$  be $L$-spaces, and $\omega$ be a modulus of continuity. For a function $f\in H^\omega (T,X)$ and an operator  $\Lambda\colon B(T, X)\to Y$, the mapping
$$
f\mapsto \Lambda(\chi \cdot f)
$$
can be considered as a model of a measuring device.
\begin{definition}
    For $\lambda \in\mathcal{H}(X,Y)$ a sequence $\{x_n\}\subset X$ is called $\lambda$-extremal, if $h(x_n, \theta_X)= 1$ for all $n\in \NN$, and $h_Y(\lambda(x_n),\theta_Y)\to 1$ as $n\to\infty$.
\end{definition}

\begin{theorem}\label{th::measuringDeviceDeviation}
 Let $X$ and $Y$ be $L$-spaces,  $(T,r)$ be a metric compact,  $\lambda\in \mathcal{H}(X,Y)$ and $\varphi\in \mathcal{M}(T)$. Let also  $\chi\in B(T,\RR)$ be a non-negative function  such that $\varphi(\chi)>0$, $\Lambda\in \mathcal{L}(\lambda,\varphi)$ and $t\in T$. For each  $f\in H^\omega(T,X)$
$$
h_Y\left( \lambda(f(t)), \frac 1{\varphi(\chi)}\Lambda(\chi f)\right) 
  \le
  \frac 1{\varphi(\chi)}\varphi(\omega(r(t,\cdot))\chi(\cdot)).
$$
If $X$ is isotropic and $X^{\rm c}$ contains a $\lambda$-extremal sequence, then the inequality is sharp.
\end{theorem}
\begin{proof}
\begin{gather*}
h_Y\left( \lambda(f(t)), \frac 1{\varphi(\chi)}\Lambda(\chi f)\right)
=
h_Y\left(\frac 1{\varphi(\chi)}\Lambda(\chi f(t)), \frac 1{\varphi(\chi)}\Lambda(\chi f)\right)
\\ \leq
\frac 1{\varphi(\chi)}\varphi(h_Y[\lambda(\chi f(t)),\lambda(\chi f)])
\leq
\frac 1{\varphi(\chi)}\varphi(h_X[f(t),f(\cdot)]\chi(\cdot))
\\ \le
\frac 1{\varphi(\chi)}\varphi(\omega(r(t,\cdot))\chi(\cdot)).
\end{gather*}
Let $X$ be isotropic and $\{x_n\}\subset X^{\rm c}$ be a $\lambda$-extremal sequence. Then due to Lemma~\ref{l::LspaceValuedFunc} each function $f_n =  \omega(r(t,\cdot))\cdot x_n$ belongs to $H^\omega(T,X)$,  $f_n(t) = \theta$, $n\in\NN$, and \begin{gather*}
h_Y\left( \lambda(f_n(t)), \frac 1{\varphi(\chi)}\Lambda(\chi f_n)\right)
=
h_Y\left( \theta, \frac 1{\varphi(\chi)}\Lambda(\chi\omega(r(t,\cdot)) x_n)\right)
\\=
h_Y\left( \theta, \frac{\varphi(\omega(r(t,\cdot))\chi)}{\varphi(\chi)}\lambda(x_n)\right)
\to
\frac {\varphi(\omega(r(t,\cdot))\chi)}{\varphi(\chi)}, n\to\infty.
\end{gather*}
\end{proof}
Applying Theorem~\ref{th::measuringDeviceDeviation} to integral operator~\eqref{integralLambda}, we obtain the following Ostrowski-type (see~\cite{Ostrowski38}) inequality. This corollary contains a number of Ostrowski type inequalities for non-real valued functions, in particular for fuzzy sets (see~\cite{Anastassiou03}), random processes (see~\cite{Drozhzhina} and \cite[Section~5.5]{Babenko21}), and multi-valued function (see \cite{Babenko16a}). Some other abstract inequalities of this type and further references can be found in~\cite{Kovalenko23e,Babenko24AbstractOstrowski}.
 \begin{corollary}\label{c::OstrowskiIneq}
 Let $\mu$ be a Borel measure on a metric compact $(T,r)$ and $Q\subset T$ be a compact set with $\mu(Q)>0$. Then for all $t\in T$ and $f\in H^\omega (T,X)$, one has
 $$
h_X\left( P(f(t)), \frac 1{\mu Q}\int_Qf(s)d\mu(s)\right)
\le   \frac 1{\mu Q}\int_Q \omega\left(r(t, s)\right)d\mu(s),
$$
where $P$ is the convexifying operator. If $X$ is isotropic and $X^{\rm c}\neq \{\theta\}$, then the inequality is sharp.
 \end{corollary}
 \begin{proof}
     Observe that due to Remark~\ref{r::P(x)} for each $x\in X^{\rm c}$ with $h_X(x,\theta) = 1$ the constant sequence $\{x\}$ is $P$-extremal. Thus it suffices to apply Theorem~\ref{th::measuringDeviceDeviation} with $\chi = \chi_Q$ being the characteristic function of the set $Q$ to integral operator~\eqref{integralLambda}.
 \end{proof}

\begin{definition}
Let a non-negative function $\chi\in B(T,\RR)$ be given. A functional $\varphi\colon B(T,\RR)\to\RR$ is called $\chi$--averaging, if $\varphi(\chi)\neq 0$ and for all non-negative functions $u\in B(T,\RR)$ the function
    $$
    \psi^\chi(\cdot)=u(\cdot)-\frac 1{\varphi(\chi)}\varphi(u\chi)
    $$
    satisfies
    \begin{equation}\label{mean}
        \varphi(\psi^\chi_+\chi)=\varphi(\psi^\chi_-\chi).
    \end{equation}
\end{definition}
\begin{theorem}\label{th::dualMeasuringDeviceDeviation}
 Let $X$ and $Y$ be $L$-spaces,  $(T,r)$ be a metric compact,  $\lambda\in \mathcal{H}(X,Y)$ and $\varphi\in \mathcal{M}(T)$. Let also $\chi \in B(T,\RR)$ be a non-negative function such that $\varphi(\chi)>0$, $\Lambda\in \mathcal{L}(\lambda,\varphi)$ and $t\in T$. Then
\begin{equation}\label{meanzeroest}
\sup_{\substack{f\in H^\omega(T,X),\\ \Lambda(\chi f)=\theta}}h_Y\left( \lambda(f(t)),\theta\right)
\le
\frac 1{\varphi(\chi)}\varphi(\omega(r(t,\cdot))\chi(\cdot)).
\end{equation}
If $X$ is isotropic, $\varphi$ is $\chi$--averaging, 
\begin{equation}\label{LambdaXpm}
\Lambda(u_x)=\lambda(x)\varphi(u_+)+\lambda(x')\varphi(u_-)
\end{equation}
for all $x\in X^{\rm c}\cap X^{\rm inv}$ and $u\in B(T,\RR)$, and the set
\begin{equation}\label{lambdaSup}
\tilde{X} = \left\{x\in   X^{\rm inv}\cap X^{\rm c}\colon\lambda(x)\in Y^{\rm inv}\cap Y^{\rm c},\; \lambda(x') = (\lambda(x))'\right\} 
\end{equation}
contains a $\lambda$-extremal sequence, then the inequality turns into equality.
\end{theorem}
\begin{proof}
Inequality~\eqref{meanzeroest} follows from Theorem~\ref{th::measuringDeviceDeviation}. We prove that it becomes an equality under the additional conditions specified in the theorem.
Let 
\begin{equation}\label{psiQ}
    \psi^\chi(\cdot)=\omega(r(t,\cdot))-\frac 1{\varphi(\chi)}\varphi(\omega(r(t,\cdot))\chi(\cdot))\in H^\omega(T,\RR).
\end{equation}
Equality~\eqref{mean} holds, since $\varphi$ is $\chi$--averaging. Then for $x\in \tilde{X}$, $h_X(x,\theta)\leq 1$, and $f^\chi = \psi^\chi_x\in H^\omega(T,X)$,
\begin{gather*}
\Lambda(\chi f^\chi) = \Lambda((x\psi^\chi_++x'\psi^\chi_-)\chi)=\lambda(x)\varphi(\psi^\chi_+\chi)+\lambda(x')\varphi(\psi^\chi_-\chi)
\\
=\lambda(x)\varphi(\psi^\chi_+\chi)+(\lambda(x))'\varphi(\psi^\chi_+\chi)=(\lambda(x)+(\lambda(x))')\varphi(\psi^\chi_+\chi)=\theta.
\end{gather*}
Note that 
$
\psi^\chi(t)
=
-\frac 1{\varphi(\chi)}\varphi(\omega(r(t,\cdot))\chi(\cdot))\le 0$, and hence
$$
f^\chi(t) = \psi^\chi_x(t)=\frac 1{\varphi(\chi)}\varphi(\omega(r(t,\cdot))\chi)x'.
$$
Thus 
\begin{gather*}
\sup_{\substack{f\in H^\omega(T,X)\\ \Lambda(\chi f)=\theta}}h_X\left( \lambda(f(t)),\theta\right)
\ge 
h_X(\lambda (\psi^\chi_x(t)), \theta)
\\
=
h_X\left(\lambda \left(\frac 1{\varphi(\chi)}\varphi(\omega(r(t,\cdot))\chi)x'\right),\theta\right)
=
 h_X(\lambda(x'),\theta)\frac {\varphi(\omega(r(t,\cdot))\chi)}{\varphi(\chi)}.
\end{gather*}
Taking into account existence of a $\lambda$-extremal sequence in set~\eqref{lambdaSup}, we obtain the required estimate from below.
\end{proof}
Observe that for integral operator~\eqref{integralLambda}, equality~\eqref{LambdaXpm} holds due to Lemma~\ref{l::LspaceValuedFunc}, the functional $\varphi$ is $\chi_Q$--averaging on any set $Q\subset T$ (here again $\chi_Q$ is the characteristic function of the set $Q$). The convexifying operator $P$ is the identity map on set~\eqref{lambdaSup} 
due to Lemmas~\ref{l::convexifyingOperator} and~\ref{l::inverseIsConvex}, and thus set~\eqref{lambdaSup} contains a $P$-extremal sequence provided
\begin{equation}\label{xcXinv}
X^{\rm inv}\cap X^{\rm c}\neq\{\theta\}.
\end{equation}
Hence  applying Theorem~\ref{th::dualMeasuringDeviceDeviation} we obtain the following result. 
\begin{corollary}%\label{c::valueRecoveryEstimate}
 Let $\mu$ be a Borel measure on a metric compact $(T,r)$, $t\in T$, and $Q\subset T$ be a compact set with $\mu(Q)>0$. Then
  $$
 \sup_{\substack{f\in H^\omega (T,X),\\ \int_Qf(t)d\mu(t)=\theta}}h_X(P(f(t)),\theta)\le\frac 1{\mu Q}\int_Q \omega\left(r(t, s)\right)d\mu(s).
 $$
If $X$ is isotropic and~\eqref{xcXinv} holds, then the inequality becomes equality.
\end{corollary}

 \section{Optimal Recovery Problems}\label{s::optimalRecoveryProblems}
 \subsection{One Auxiliary Lemma}
  We need the following well known estimate for the value of the optimal recovery~\eqref{errorOfRecovery}. It follows e.g., from Theorem~1 and Lemma~2 in~\cite[\S 1.2]{Zhensykbaev}. We give its short proof for completeness.
 \begin{lemma}\label{l::errorOfRecoveryFromBelow}
If $f,g\in W$ are such that $I(f) = I(g)$, then 
 $$ {\mathcal E}(\Lambda,W, I, X)\geq \frac 12 h_X(\Lambda(f),\Lambda(g)).$$
\end{lemma}
\begin{proof} We have
\begin{gather*}
    \sup_{w\in W} h_X(\Lambda(w), \Phi(I(w)))
    \geq
    \max\left\{h_X(\Lambda(f), \Phi(I(f))), h_X(\Lambda(g), \Phi(I(g)))\right\}
   \\ \geq 
     \frac12\left(h_X(\Lambda(f), \Phi(I(f))) + h_X(\Lambda(g), \Phi(I(f)))\right)
\geq 
      \frac12 h_X(\Lambda(f), \Lambda(g)),
\end{gather*}
as desired.
\end{proof}
  \subsection{Recovery of the Integral}
 Let $(T,r)$ be a metric compact set, $Q\subset T$ be also a compact, $\mu$ be a Borel measure on $T$ and $x_1,\ldots, x_n\in T$. In this subsection we solve the problem of the optimal recovery of operator~\eqref{integralLambda}   on the class $H^\omega(T,X)$ based on information operator~\eqref{informationOperator1}.
 For $i =1, \ldots, n$ set
$$
   \tilde{T}_i = \{t\in T\colon r(t,x_i)\leq r(t,x_j), j\neq i\},
$$
\begin{equation}\label{Ti}
T_1 = \tilde{T}_1, \text{ and } T_j = \tilde{T}_j\setminus\bigcup_{i=1}^{j-1}\tilde{T}_i \text{ for } j=2,\ldots, n.    
\end{equation}

 \begin{theorem} If~\eqref{xcXinv} holds, then  for the optimal recovery of operator~\eqref{integralLambda} based on the information given by~\eqref{informationOperator1}, one has
 $$\mathcal{E}(\Lambda, H^\omega(T,X), I, X)
 =
 \int_Q
\min_{i=1,\ldots ,n}\omega(r(x_i,t))d\mu(t).
 $$
 The optimal method of recovery is
 $$
 \Phi^* (I(f))=\sum_{i=1}^n P(f(x_i))\mu (T_i\cap Q).
 $$
  \end{theorem}
  \begin{proof}
  For all $f\in H^\omega(T,X)$, applying inequality~\eqref{distBetweenSums} and Corollary~\ref{c::OstrowskiIneq} to each of the sets $T_i\cap Q$, we obtain 
\begin{gather*}
  h_X\left(\int_{Q}f(s)d\mu(s), \Phi^* (I(f))\right)
  \\=
  h_X\left(\sum_{i=1}^n\int_{T_i\cap Q}f(s)d\mu(s), \sum_{i=1}^n P(f(x_i))\mu (T_i\cap Q)\right)
  \\ \leq 
  \sum_{i=1}^n\int_{T_i\cap Q}\omega(r(x_i,t))d\mu(t)
  =
  \int_Q
\min_{i=1,\ldots ,n}\omega(r(x_i,t))d\mu(t).    
\end{gather*}
On the other hand, applying Lemma~\ref{l::errorOfRecoveryFromBelow} to the functions  
$$
\min_{i=1,\ldots ,n}\omega(r(x_i,\cdot))\cdot  y \text{ and } \min_{i=1,\ldots ,n}\omega(r(x_i,\cdot))\cdot y', 
 $$
 where $y\in X^{\rm c}\cap X^{\rm inv}$ and $h_X(y,\theta)=1$, using Lemma~\ref{l::distBetweenInverseElems} we obtain 
 \begin{gather*}
 \mathcal{E}(\Lambda, H^\omega(T,X), I, X)
 \geq 
 \frac 12\int_Q\min_{i=1,\ldots ,n}\omega(r(x_i,t))d\mu(t) h_X(y,y')
\\= 
\int_Q
\min_{i=1,\ldots ,n}\omega(r(x_i,t))d\mu(t).
\end{gather*}
  \end{proof}

\subsection{Recovery of Operator~\eqref{valueOperator}.}
   Let $(T,r)$ be a metric compact, $X$ and $Y$ be $L$-spaces, $\lambda\in \mathcal{H}(X,Y)$ and $\varphi\in \mathcal{M}(T)$. Let also non-negative functions $\chi_1,\ldots,\chi_n\in B(T,\RR)$ be such that  $\varphi(\chi_{i})>0$, $i=1,\ldots, n$. In this subsection for a fixed $t\in T$ we solve the problem of optimal recovery of operator~\eqref{valueOperator} 
   on the class $H^\omega(T,X)$ based on the information operator~\eqref{generalInformationalSet}.
\begin{theorem}
Let $t\in T$ be given. For the optimal recovery of operator~\eqref{valueOperator} based on the information given by~\eqref{generalInformationalSet}, one has
 \begin{equation}\label{valueRecoveryUpperEstimate}
 \mathcal{E}(\Lambda, H^\omega(T,X), I, Y)
 \leq 
 \min_{i=1,\ldots,n}\frac 1{\varphi(\chi_{i})}\varphi(\omega(r(t,\cdot))\chi_{i}(\cdot)).
\end{equation}
 If $X$ is isotropic, set~\eqref{lambdaSup} contains a $\lambda$-extremal sequence,   minimum on the right-hand side of~\eqref{valueRecoveryUpperEstimate} is attained on the index $i^*$, $\varphi$ is $\chi_{i^*}$--averaging, and the operator $\Lambda_{i^*}$ satisfies~\eqref{LambdaXpm}, then inequality~\eqref{valueRecoveryUpperEstimate} becomes equality. In this case the optimal method of recovery is
 $$
 \Phi^* (I(f))= \varphi(\chi_{i^*})\Lambda_{i^*}(\chi_{i^*} f).
 $$
\end{theorem}
\begin{proof}
Inequality~\eqref{valueRecoveryUpperEstimate} follows from Theorem~\ref{th::measuringDeviceDeviation}. The estimate from below under additional conditions of the theorem can be obtained as follows. For $x\in \tilde{X}$ (see~\eqref{lambdaSup}) consider the functions 
$$
f = \psi^{\chi_{i^*}}_+x+\psi^{\chi_{i^*}}_-x' \text{ and } g = \psi^{\chi_{i^*}}_+x'+\psi^{\chi_{i^*}}_-x,
$$
where $\psi^{\chi_{i^*}}$ is built according to~\eqref{psiQ} with $\chi$ substituted by $\chi_{i^*}$. Application of Lemma~\ref{l::errorOfRecoveryFromBelow} implies the required estimate from below under an appropriate choice of $x\in\tilde{X}$.
\end{proof}

If $\mu$ is a Borel measure on $T$, $Q_i\subset T$ are compacts with $\mu(Q_i) > 0$, $\chi_i = \chi_{Q_i}$,  $i=1,\ldots, n$, $\varphi(u) = \int_{T}u(s)d\mu(s)$,  $\lambda=P$ is the convexifying operator and $\Lambda_i(f) = \int_Tf(s)d\mu(s)$, then we obtain the following corollary.
\begin{corollary}\label{c::valueRecovery}
Let $t\in T$ be given. For the optimal recovery of the operator $\Lambda(f) = P(f(t))$ based on the information 
$$
 I(f) 
   = 
   \left(\frac{1}{\mu(Q_1)}\int_{Q_1}f(s)d\mu(s), \ldots, \frac{1}{\mu(Q_n)}\int_{Q_n}f(s)d\mu(s) \right), 
$$
one has
 \begin{equation}\label{valueRecoveryUpperEstimate2}
 \mathcal{E}(\Lambda, H^\omega(T,X), I, X)
 \leq 
 \min_{i=1,\ldots,n}\frac 1 {\mu (Q_i)}\int_{Q_i}\omega(r(t,s))d\mu(s).
\end{equation}
 If $X$ is isotropic, and condition~\eqref{xcXinv} holds, then inequality~\eqref{valueRecoveryUpperEstimate2} becomes equality. In this case the optimal method of recovery is
 $$
 \Phi^* (I(f))= \frac{1}{\mu(Q_{i^*})}\int_{Q_{i^*}}f(s)d\mu(s),
 $$ 
 where $i^*$ is such that the minimum in~\eqref{valueRecoveryUpperEstimate2} is attained on the index $i^*$.
\end{corollary}
  
\subsection{Recovery of the Convexifying Operator}
Finally, we consider the problem of optimal recovery  of the convexifying operator $\Lambda = P$ based on information operator~\eqref{meanValuesInformation} in the uniform norm.
\begin{theorem}
Let $T\subset \RR^d$, $d\in\NN$, be a convex compact set, $\mu$ be the Lebesgue measure and $r$ be the Euclidean metric. For the optimal recovery of the convexifying operator $P$ in the uniform norm based on the information given by~\eqref{meanValuesInformation},  one has
 \begin{equation}\label{functionRecoveryUpperEstimate}
 \mathcal{E}(P, H^\omega(T,X), I, B(T,X))
 \leq 
 \max_{t\in T} \min_{i=1,\ldots,n}\frac 1 {\mu B_i}\int\limits_{B_i}\omega(r(t,s))d\mu(s).
\end{equation}
 If $X$ is isotropic and~\eqref{xcXinv} holds, then inequality~\eqref{functionRecoveryUpperEstimate} becomes equality for all small enough radii $\varepsilon>0$. In this case the optimal method of recovery is
 $$
 \Phi^* (I(f))(t)
 =
 \sum_{i=1}^n\left(\frac{1}{\mu(B_i)}\int_{B_i}f(s)d\mu(s)\right)\chi_{T_i}(t),
 $$
 where the sets $T_i$ are defined in~\eqref{Ti}, $i=1,\ldots, n$.
\end{theorem}
\begin{proof}
Estimate~\eqref{functionRecoveryUpperEstimate} follows from Corollary~\ref{c::valueRecovery}. To obtain the estimate from below, we consider the case $X=\RR$ first.

In the set $T$ find a ball with maximal radius that does not contain points $x_1,\ldots, x_n$ in its interior. Let $\overline{x}$ be its center and $R$ be its radius. Let the radius $\varepsilon$ of the balls $B_i$, $i=1,\ldots, n$, be so small,  that for $i=1,\ldots, n$  either $r(x_i,\overline{x}) = R$, or $r(x_i,\overline{x}) \geq  R+4\varepsilon$.  Consider the univariate function
 $$
 \varphi (t)=
\begin{cases}
\omega (R+\varepsilon)-\omega(t),& 0\le t\le R+\varepsilon;\\
\varphi (2(R+\varepsilon) -t),& R+\varepsilon <t\le R+3\varepsilon;\\
\varphi(R-\varepsilon),& t>R+3\varepsilon.
\end{cases}
$$
For definiteness we assume that $r(x_1,\overline{x}) = R$. Define a function $\psi\colon \RR^d\to\RR$,
$$
\psi(s) = \varphi(r(\overline{x},s)) - \frac{1}{\mu (B_1)}\int_{B_1}\varphi(r(\overline{x},u))du.
$$
Note that $\psi(s_1) = \psi(s_2)$ provided $r(\overline{x},s_1) =r(\overline{x},s_2)$, and the set $\{r(\overline{x},s)\colon \psi(s) = 0\}$ contains exactly two numbers, both from the interval $(R-\varepsilon, R+3\varepsilon)$. Denote by $\overline{R}$ the larger of the two numbers and define the function $\Psi\colon \RR^d\to\RR$,
$$
\Psi(s) 
=
\begin{cases}
\psi(s),& r(\overline{x}, s)\leq \overline{R}, \\
0,& r(\overline{x}, s)> \overline{R}.
\end{cases}
$$
Since $\varphi\in H^\omega([0,R+3\varepsilon],\RR)$, we obtain that $\psi,\Psi\in H^\omega(T,\RR)$. Moreover, if $x\in\RR^d$ is such that either $r(\overline{x},x) = R$, or $r(\overline{x},x) \geq R+4\varepsilon$, then due to construction of $\Psi$,
$
\int_{\{r(u,x)\leq \varepsilon\}}\Psi(u)du = 0.
$
Thus $I(\Psi) = (0,\ldots, 0)$. Moreover, 
\begin{gather*}
\Psi(\overline{x}) 
= 
\varphi(0) - \frac{1}{\mu (B_1)}\int_{B_1}\varphi(r(\overline{x},u))du
\\=
\omega(R+\varepsilon) - \omega(R+\varepsilon) + \frac{1}{\mu( B_1)}\int_{B_1}\omega(r(\overline{x},u))du 
\\=
\max_{x\in T} \min_{i=1,\ldots, n}\frac 1 {\mu B_i}\int_{B_i}\omega(r(x,s))d\mu(s).
\end{gather*}
Finally, for $y\in X^{\rm c}\cap X^{\rm inv}$, $h_X(y,\theta) = 1$, consider the functions $\Psi_y$ and $\Psi_{y'}$. Due to Lemma~\ref{l::errorOfRecoveryFromBelow} they provide the required estimate from below.
\end{proof}

\bibliographystyle{elsarticle-num}
\bibliography{bibliography}

 \end{document}